\documentclass[reqno, 12pt]{amsart}

\usepackage{amssymb}
\usepackage{bm}
\usepackage{graphicx}
\usepackage{psfrag}
\usepackage{hyperref}
\hypersetup{colorlinks=true, linkcolor=blue, citecolor=red, urlcolor=wine}
\usepackage{color}
\usepackage{url}
\usepackage{algpseudocode}
\usepackage{fancyhdr}
\usepackage{xy}
\input xy
\xyoption{all}
\usepackage{stmaryrd}
\usepackage{calrsfs}
\usepackage[dvipsnames]{xcolor}

\usepackage{amsmath}
\usepackage{geometry}

\usepackage{textcomp}

\usepackage{amssymb}

\usepackage{epsfig}

\usepackage{color}

\newtheorem{thm}{Theorem}[section]
\newtheorem{lem}[thm]{Lemma}

\newtheorem{remark}[thm]{Remark}

\newtheorem{prop}[thm]{Proposition}

\newtheorem{conj}[thm]{Conjecture}

\numberwithin{equation}{section}

\newcommand{\floor}[1]{\left\lfloor #1 \right\rfloor}

\newcommand{\df}{\overset{\rm{def}}{=}}

\newcommand{\R}{{\mathbb{R}}}

\newcommand{\Z}{{\mathbb{Z}}}

\newcommand{\N}{{\mathbb{N}}}

\newcommand{\e}{\varepsilon}

\DeclareMathOperator{\BBA}{\widetilde{d-BA}}
\DeclareMathOperator{\BA}{\widetilde{2-BA}}
\DeclareMathOperator{\Ba}{2-BA}
\DeclareMathOperator{\BBa}{d-BA}
\DeclareMathOperator{\ba}{BA}

\calclayout

\begin{document}

\begin{abstract}
In this paper we study the classical Schmidt game on two families of sets: one related to frequencies of digits in base-$2$ expansions, and one connected to the set of the badly approximable numbers. Namely, we describe some nontrivial winning and losing parameters $(\alpha, \beta)$ for these sets.
\end{abstract}

\title[On Nontrivial Winning and Losing Parameters of Schmidt Games]{On Nontrivial Winning and Losing Parameters of Schmidt Games}
\author[Neckrasov]{Vasiliy Neckrasov}
\address{Brandeis University, Waltham, MA 02453, USA}
\email{vneckrasov@brandeis.edu}

\author[Zhan]{Eric Zhan}
\address{Massachusetts Institute of Technology, Cambridge, MA 02139, USA}
\email{ezhan@mit.edu}
\maketitle

\section{Set-up and Introduction: Schmidt Games}\label{sg}

\bigskip

Wolfgang Schmidt first introduced Schmidt games in \cite{S} in order to study the set of badly approximable numbers. These games involve two players, Alice and Bob and two parameters $\alpha,\beta\in (0,1)$. A subset $S$ of a metric space is called $(\alpha,\beta)$-winning if Alice has a strategy guaranteeing that a certain point, constructed as a result of the infinite series of moves, lies in $S$. Of particular interest are sets that are $(\alpha,\beta)$-winning for all $\beta$ and a fixed $\alpha$; those are called $\alpha$-winning sets. Schmidt proved that $\alpha$-winning sets have full Hausdorff dimension, and a countable intersection of $\alpha$-winning sets is $\alpha$-winning. The technique of Schmidt games has been frequently used to prove that certain sets arising in number theory and dynamical systems have full Hausdorff dimension.  Many of the applications of Schmidt games deal with various stronger modifications of Schmidt's original construction. These modifications include the McMullen's strong and absolute winning games \cite{mcmullen}, with an additional property that the winning property is preserved by quasisymmetric homeomorphisms; hyperplane absolute winning game \cite{HAW}, where the winning property holds under $C^1$ diffeomorphisms; potential game \cite{FSU18}, Cantor game \cite{BH17}, and Hausdorff/packing games  \cite{variational} which give a tool to precisely calculate Hausdorff/packing dimensions. For a detailed survey describing the setups of some of the aforementioned games and connections between them, see \cite{bhns18}. However, in this paper, we restrict our attention to the $(\alpha,\beta)$-game introduced in \cite{S}.

Here, we briefly describe the setup of the classical Schmidt $(\alpha, \beta)$-game. Throughout this paper we work with the Schmidt game on $\mathbb{R}$, so for simplicity we give the definitions only in this special case.

Denote by $I$ the open unit square: $$I \df \{(\alpha,\beta): 0 < \alpha,\beta < 1\} = (0,1)\times(0,1)\,.$$
Then, fix a set $S \subset \R$ and $(\alpha,\beta)\in I$. Two players, Alice and Bob, can play a Schmidt game on $S$, taking alternating turns, as follows: Bob begins by choosing a closed interval $B_0$. Given an interval $B_i$, Alice chooses a closed interval $A_i \subset B_i$ such that $|A_i| = \alpha|B_i|$. Similarly, given an interval $A_i$, Bob chooses an interval $B_{i+1} \subset A_i$ such that $|B_{i+1}| = \beta|A_i|$. Let
$$
x_\infty \df \bigcap_{k = 1}^\infty  B_k = \bigcap_{k = 1}^\infty  A_k.$$
We call $S$ $(\alpha,\beta)$\textit{-winning} if, regardless of how Bob chooses his intervals, Alice can play in such a way that $x_\infty \in S$. 

\vskip 0.2 cm

In this paper, we will often fix the set $S$ and explore the games for different pairs $(\alpha, \beta)$. To do this, we will utilize a little bit different language.

Let's fix the set $S$. A point $(\alpha, \beta)\in I$ is called \textit{winning} for $S$ if $S$ is $(\alpha, \beta)$-winning and \textit{non-winning} for $S$ otherwise. A point $(\alpha, \beta)\in I$ is called \textit{losing} for $S$ if, regardless of how Alice chooses her intervals, Bob can play in such a way that $x_\infty \notin S$, and \textit{non-losing} otherwise.

\vskip 0.2 cm

It will be crucial for some of our proofs to know that for a given set $S$ and any parameters $(\alpha, \beta)$ one of the players always has a strategy to win; that is, we will require that all the non-losing parameters are winning, and all the non-winning parameters are losing. In this case, we will say that the Schmidt game is \textit{determined} on the set $S$. In general, Schmidt game can be not determined: see \cite{AFG23} for some constructions using the Axiom of choice. However, it is known (\cite{CFJ23} ) that the Schmidt game in $\mathbb{R}$ is determined if we assume the Axiom of determinacy AD (which contradicts the Axiom of choice). More specifically, it is known that the Schmidt game is determined on Borel sets.

\vskip 0.2 cm

For every $S$, define the \textit{Schmidt diagram} $D(S)$ of $S$ as the set of all winning $(\alpha,\beta)\in I$. In his original paper, Schmidt proved that there are two ``trivial zones'' of Schmidt diagrams:

\begin{lem}\label{trivial}\cite{S}
\begin{itemize}
\item[\rm(a)]   If $S$ is dense, then $D(S)$ contains the set $$\check D :=  \big \{(\alpha,\beta) \in I : \beta \ge (2-\alpha)^{-1}\big\},$$ with equality   if $S$ is   countable (see Figure~\ref{Fig:Data1}).
\smallskip

\item[\rm(b)]   If $S\subsetneq\R$, then $D(S)$ is contained in   the set $$\hat D := \big\{(\alpha,\beta) \in I : \beta > 2 - \alpha^{-1}\big\},$$ with equality   if the complement of $S$ is   countable (see Figure~\ref{Fig:Data2}).
\end{itemize}
\end{lem}

\begin{figure}[!htb]
   \begin{minipage}{0.45\textwidth}
     \centering
     \includegraphics[width=.7\linewidth]{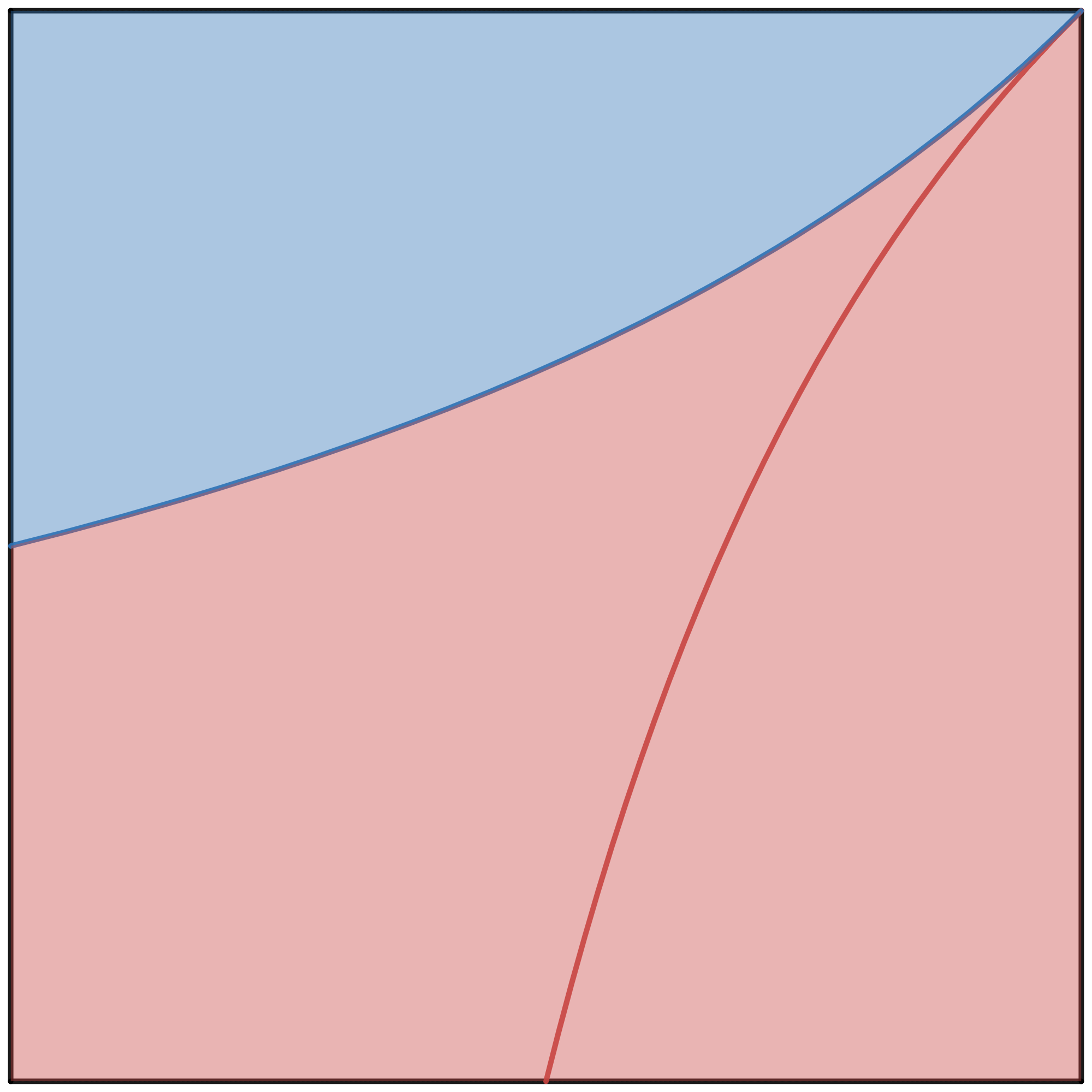}
     \caption{$\check D$ (blue)}\label{Fig:Data1}
   \end{minipage}\hfill
   \begin{minipage}{0.45\textwidth}
     \centering
     \includegraphics[width=.7\linewidth]{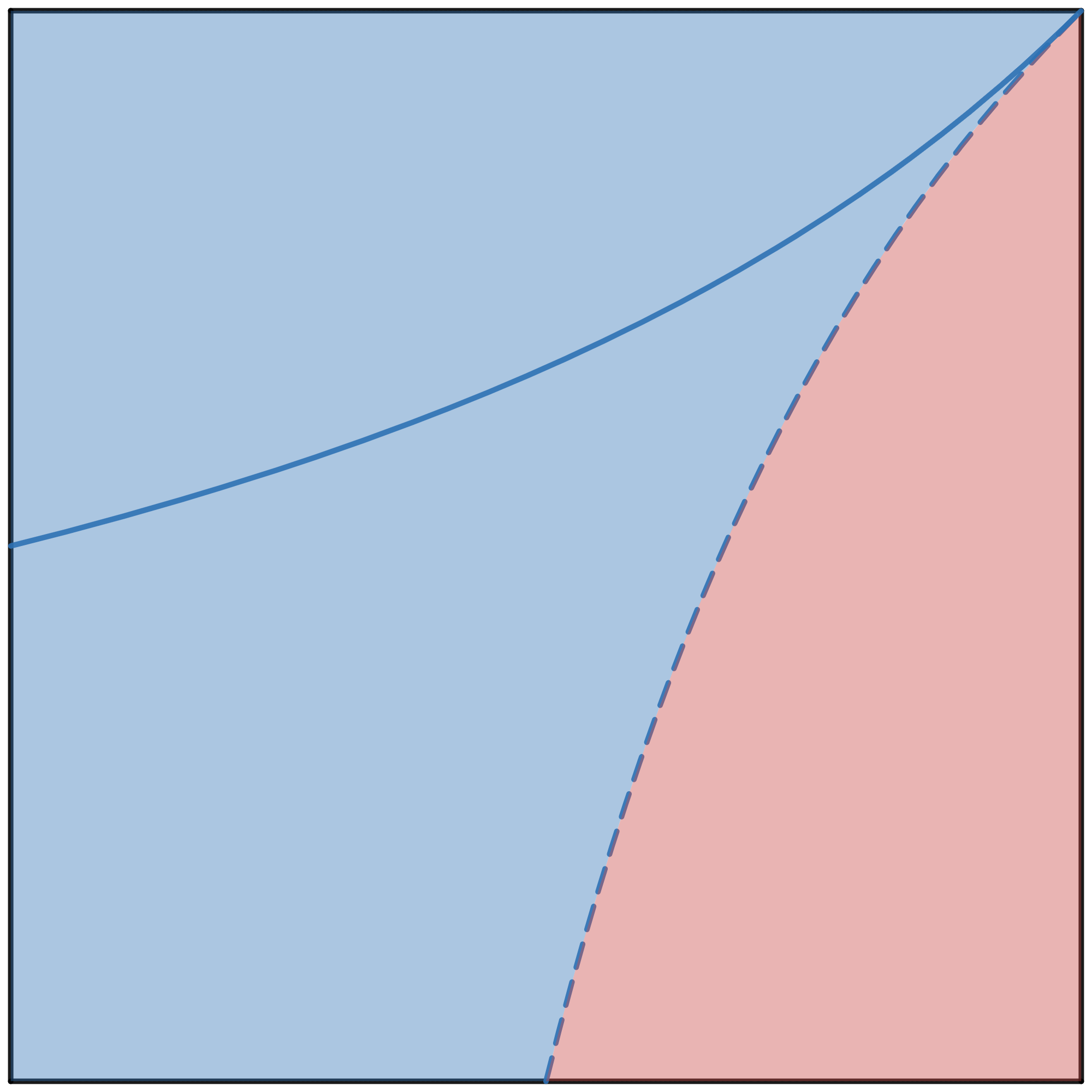}
     \caption{$\hat D$ (blue)}\label{Fig:Data2}
   \end{minipage}
\end{figure}

There are only four sets that are known to coincide with $D(S)$ for some $S$: $\emptyset$, $\check D$, $\hat D$, and $I$. It is clear that this is far not the complete list of the possible Schmidt diagrams, see \cite{F} for an example of the set which is proven to have a nontrivial diagram, however, no diagram $D(S)$ different from the four diagrams discussed above was completely described. It is worth noting that the problem of completely describing a nontrivial Schmidt diagram was investigated in an earlier PRIMES paper \cite{KZT}, which investigated sets relating to particular decimal expansions.

In this paper, we partially describe the nontrivial Schmidt diagrams of two families of sets, and make some conjectures about their precise form.

\vskip 0.2 cm

The structure of the paper is as follows. In Section~\ref{newres}, we formulate our main results. We describe some winning and losing conditions for families of sets connected with digit frequencies in the dyadic expansion of real numbers in Subsection \ref{freq}, and for families of sets connected with badly approximable numbers in base $B$ expansion in Subsection \ref{bbares}. 
In Section \ref{general}, we discuss some general properties of winning sets and general symmetric strategies on special families of sets, which we need to prove Theorem~\ref{thmdf}. We prove Theorem~\ref{thmdf} in Section \ref{freqproofs} and the theorems from Subsection \ref{bbares} in Section~\ref{2ba}.

\section{Main results} \label{newres}

\subsection{Digit frequencies and related sets}\label{freq}

We will investigate a construction regarding frequencies of digits in base-$2$ expansions of real numbers. These sets have been studied in the past (see \cite{V, FFNS}). 

First, some setup. Consider the base-$2$ expansions of the form $x = x_0,x_1x_2 \cdots$ where $x_0$ is an integer and $x_i \in \{0, 1\}$ are the digits in the base-$2$ expansion of $x$. In this paper, if $x$ can be represented by two separate expansions, we take the lexicographically minimal expansion (for example, $0.100\ldots$ would instead be written as $0.011\ldots$). We define
$$
d^+(x, j) = \limsup_{k \rightarrow \infty} \frac{\#\{1 \leq i \leq k : x_i = j\}}k
$$
and similarly,
$$
d^-(x, j) = \liminf_{k \rightarrow \infty} \frac{\#\{1 \leq i \leq k : x_i = j\}}k.
$$

Define the sets
\begin{align*}
F^+_c &:= \{ x \in \R: d^+(x, 0) > c\} ,\\
F^-_c &:= \{ x \in \R: d^-(x, 0) > c\}.
\end{align*}

It is known (\cite{Bor}) that the set $\{ x \in \mathbb{R}: \,\, d^-(x, j) = d^+(x, j) = \frac{1}{2} \}$ has full Lebesgue measure for any $j \in \{0, 1 \}$. In particular, the sets $F^{\pm}_c$ have full Lebesgue measure if $c < \frac{1}{2}$ and measure zero otherwise. Moreover, Besicovitch (\cite{B35}; see also \cite{Olsen} for generalizations) showed that the Hausdorff dimension of $F^{-}_c$ for $c > 1$ is equal to $-c \log_2c - (1-c) \log_2(1-c)$. 

Winning and losing parameters of these sets were studied before. The following fact is not stated, but follows from the proof of Theorem 2.15 in \cite{FFNS}.

\begin{prop}\cite{FFNS}
Let $c \in (0, 1)$. The set $F^+_c$ is winning for $\log(8\alpha)/\log(\alpha\beta) > c$ and losing for $\log(\alpha/8)/\log(\alpha\beta) \leq c$.
\end{prop}

The results of our research expand on that bound:

\begin{thm}\label{thmdf}
  Let $c \in (0, 1)$, and suppose  $\log_2{\alpha\beta} \notin \mathbb{Q}$. Then, the sets $F^{\pm}_c$ are $(\alpha, \beta)$-winning if $\log(4\alpha)/\log(\alpha\beta) > c$, and $(\alpha, \beta)$-losing if $\log(\alpha/4)/\log(\alpha\beta) \leq c$. In addition: 

    \begin{enumerate}
        \item If $c < \frac{1}{2}$, then the set $F^+_c$ is $(\alpha, \beta)$-winning for any $(\alpha, \beta)$ such that $\beta > \alpha$ and $\log_2 (\alpha \beta) \notin \mathbb{Q}$.
        \item If $c \geq \frac{1}{2}$, then the set $F^-_c$ is $(\alpha, \beta)$-losing for any  $(\alpha, \beta)$ such that $alpha \geq \beta$. Here we do not require the restriction $\log_2{\alpha\beta} \notin \mathbb{Q}$.
    \end{enumerate}
\end{thm}

Let us also notice that Schmidt's original approach (see Corollary 1 of Theorem 6 in \cite{S}) automatically gives us some losing zone bounds as a function of Hausdorff dimension. Namely, for any set $S \subseteq \mathbb{R}$ the following statement holds: if $\dim_H S = s$, and $\frac{\log \lfloor \beta^{-1} \rfloor}{-\log(\alpha \beta)}>s$, then $S$ is $(\alpha, \beta)$-non-winning. Using the fact that $\dim_H F_c^{-} = -c \log_2c - (1-c) \log_2(1-c)$ for $c > \frac{1}{2}$, we automatically get the following statement:

\begin{prop}\label{estfromdim}
    Let $c > \frac{1}{2}$, and suppose $\frac{\log \lfloor \beta^{-1} \rfloor}{-\log(\alpha \beta)}>-c \log_2c - (1-c) \log_2(1-c)$. Then, the set $F_c^{-}$ is $(\alpha, \beta)$-losing.
\end{prop}

The winning and losing zones for selected values of $c$ are shown at figures ~\ref{Fig:Data3},  ~\ref{Fig:Data4},  ~\ref{Fig:Data5} and  ~\ref{Fig:Data6}. The winning zones from Theorem \ref{thmdf} are in blue, the losing zones from Theorem \ref{thmdf} are in red and the losing zones from Proposition \ref{estfromdim} are in green. We note that the losing conditions in Theorem \ref{thmdf} are always stronger that the ones from Proposition \ref{estfromdim} for small enough $\alpha$.

\vskip 0.2 cm

\begin{figure}[h!]
   \begin{minipage}{0.45\textwidth}
     \centering
     \includegraphics[width=.7\linewidth]{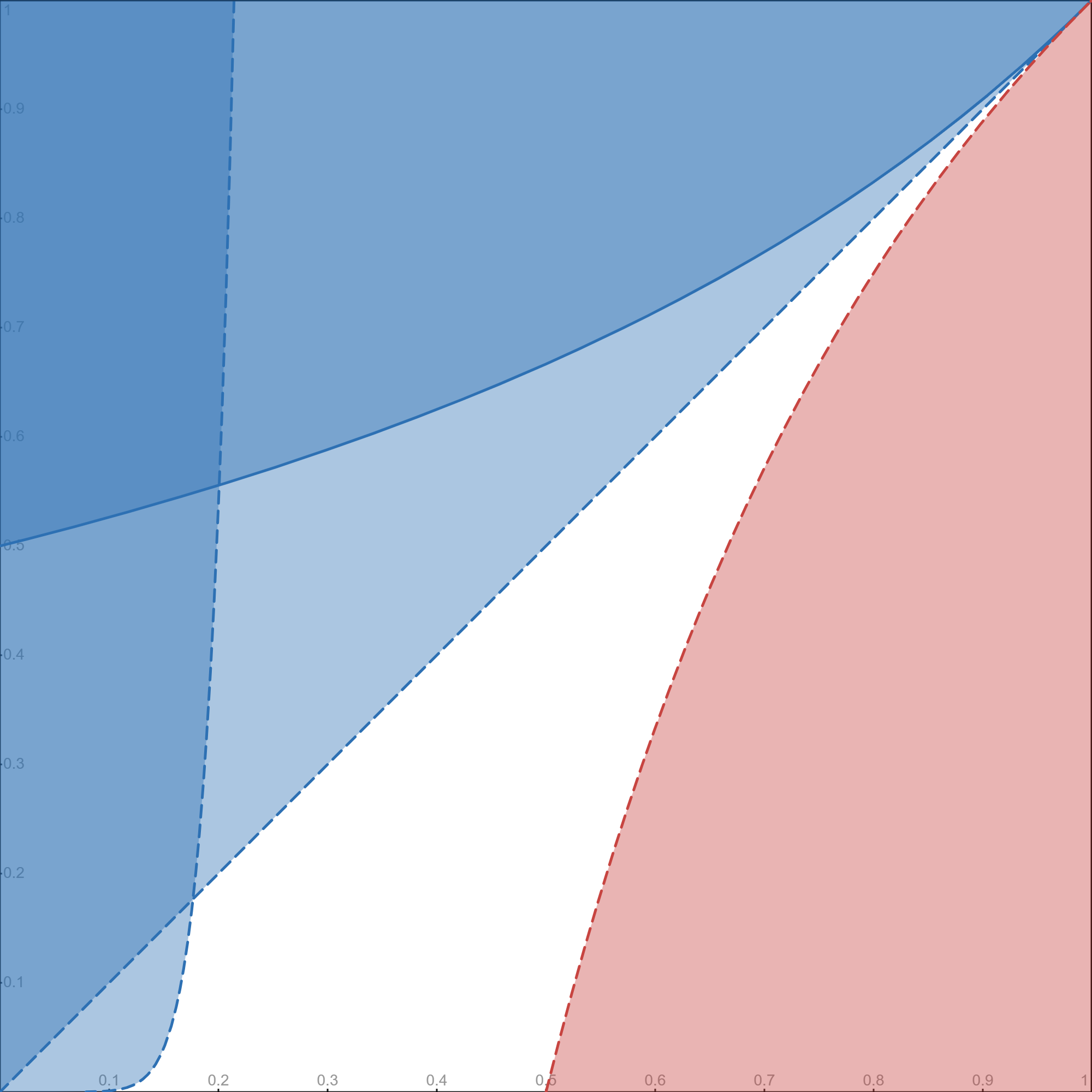}
     \caption{Bounds for $D(F^+_c)$ with $c = 0.1$}\label{Fig:Data3}
   \end{minipage}\hfill
   \begin{minipage}{0.45\textwidth}
     \centering
     \includegraphics[width=.7\linewidth]{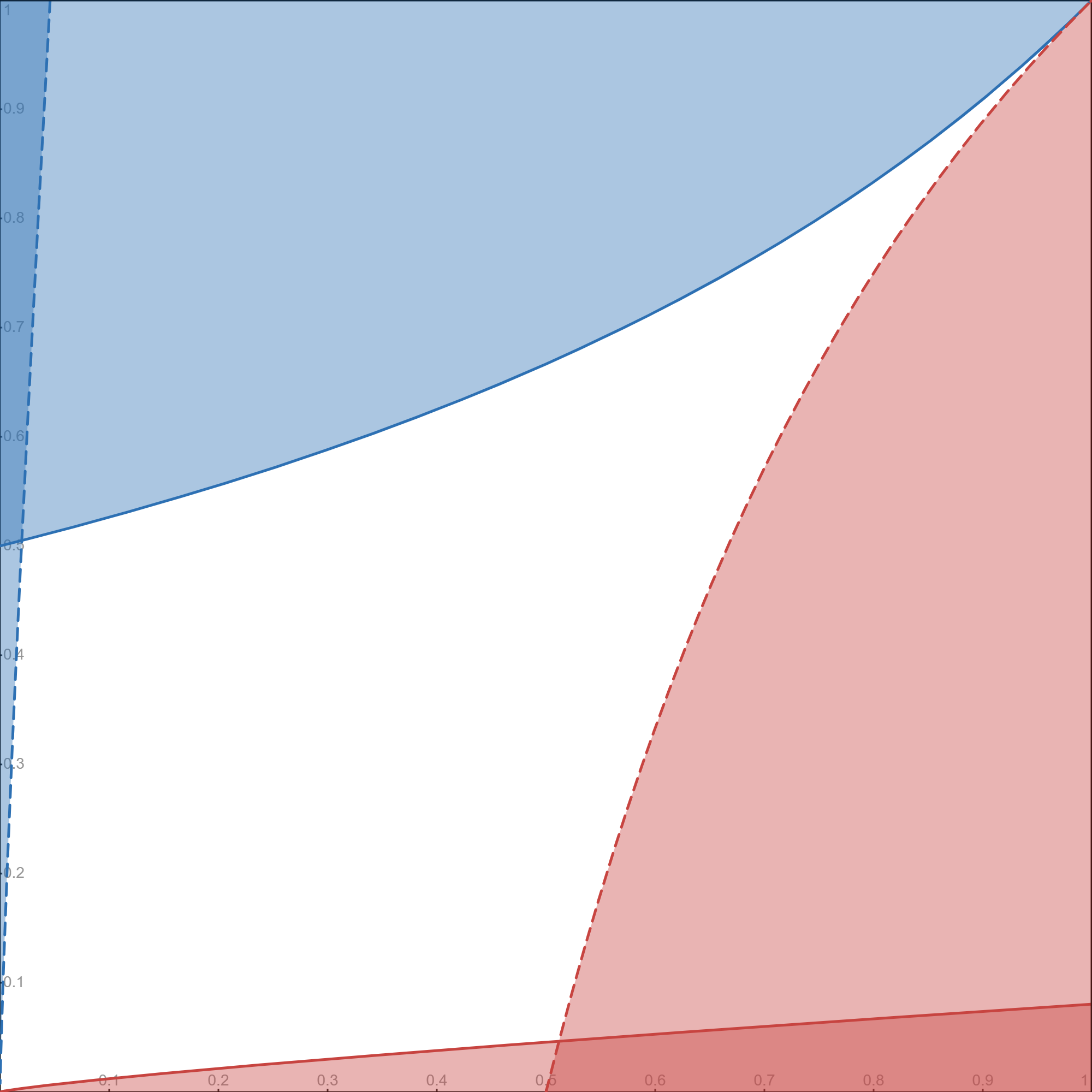}
     \caption{Bounds for $D(F^+_c)$ with $c = 0.55$}\label{Fig:Data4}
   \end{minipage}
\end{figure}
\begin{figure}[h!]
   \begin{minipage}{0.45\textwidth}
     \centering
     \includegraphics[width=.7\linewidth]{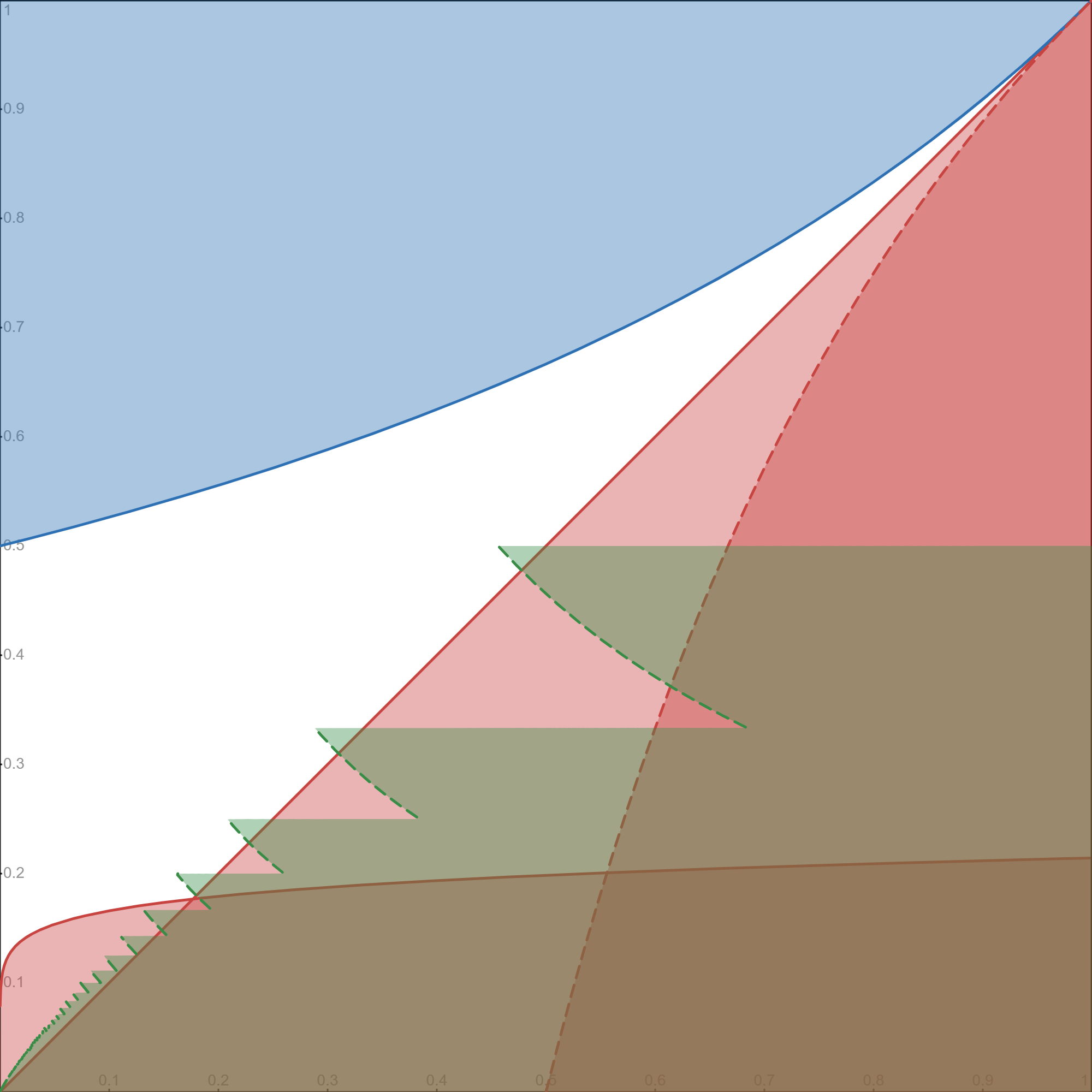}
     \caption{Bounds for $D(F^-_c)$ with $c = 0.9$}\label{Fig:Data5}
   \end{minipage}\hfill
   \begin{minipage}{0.45\textwidth}
     \centering
     \includegraphics[width=.7\linewidth]{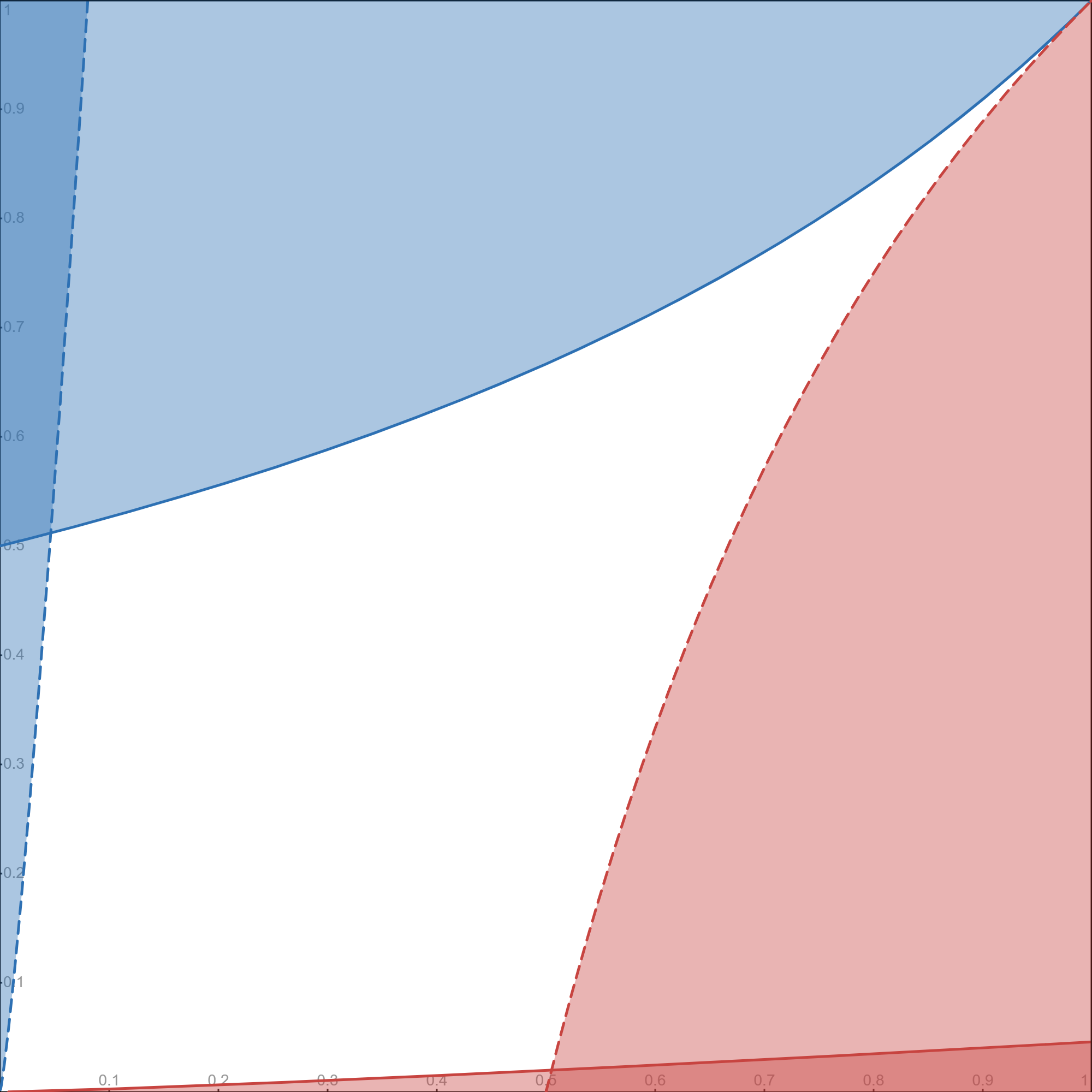}
     \caption{Bounds for $D(F^-_c)$ with $c = 0.45$}\label{Fig:Data6}
   \end{minipage}
\end{figure}



We will develop the necessary techniques and prove Theorem \ref{thmdf} in Section \ref{general}. Here, we would like to state a conjecture which our methods and results lead to.

\begin{conj}
    We believe that $D(F^{\pm}_{c}) = \{ (\alpha, \beta): \, \log(\alpha)/\log(\alpha\beta) > c \}$. In particular, we expect that $D(F^{\pm}_{\frac{1}{2}}) = \{ (\alpha, \beta): \, \alpha < \beta \}$.
\end{conj}

\vskip 0.3 cm

\subsection{d-Badly approximable numbers}\label{bbares}

The set of badly approximable numbers is an important object of study in Diophantine approximations. It is defined as follows:
$$
\ba:= \left\{x\in \R: \left|x - \frac{m}{n}\right| > \frac c{n^2}\text{ for some $c > 0$ and all }m\in\Z,\ n\in\N\right\}.
$$
It is known that the badly approximable numbers are dense and have Lebesgue measure zero. Schmidt showed in \cite{S} that they are winning for the largest possible bounds: namely, the set $\ba$ is winning for all $(\alpha, \beta) \in \hat D$. That is, $D(\ba)=\hat D$; in particular, the set $\ba$ is $\alpha$-winning for $\alpha < \frac{1}{2}$ and $\dim_H \ba = 1$. 

In a similar vein to the badly approximable numbers investigated by Schmidt, we define the set of \textit{d-Badly Approximable Numbers} as
$$
\BBa:= \left\{x\in \R: \left|x - \frac{m}{d^k}\right| > \frac c{d^k}\text{ for some $c > 0$ and all }m\in\Z,\ k\in\N\right\}
$$
and a closely related set
$$
\BBa(c):= \left\{x\in \R: \left|x - \frac{m}{d^k}\right| > \frac c{d^k}\text{ for all }m\in\Z,\ k\in\N\right\}.
$$
These definitions imply $\BBa = \bigcup_{c>0} \BBa(c)$.


It also follows from the work of Schmidt (\cite{S}) that the set $\Ba$ is winning for all $(\alpha, \beta) \in \hat D$; that is, $D(\Ba)=\hat D$.

While both $\Ba$ and $\Ba(c)$ did not produce any new Schmidt diagrams, there is a simple change we can make that will produce both nontrivial winning and nontrivial losing zones. For any integer $d \geq 2$, define
$$
\BBA(c, N):= \left\{x\in \R: \left|x - \frac{m}{d^k}\right| > \frac c{d^k}\text{ for all }m\in\Z,\ k \in \N \ \text{s.t.} \ k > N\right\}
$$
and
$$
\BBA(c) := \bigcup_{N \in \N} \BBA(c, N).
$$

We note that these sets has been previously studied in \cite{N} and \cite{P}. In particular, the Lebesgue measure of all the aforementioned sets is zero, and the precise Hausdorff dimension of $\Ba$ for any parameter $c$ is calculated in \cite{N}. Moreover, there exists such a critical value $c_0$ that the set $\Ba$ is countable if and only if $c > c_0$. This value is known to be related to Thue-Morse sequences; see Theorem 25 in \cite{N} for details.

Let us define $\gamma = 1-2\alpha+\alpha\beta$.

The following theorems provide the winning and losing conditions for $\BBA$ sets.

\begin{thm} \label{bbalos}
    Fix $0< c < \frac{1}{2}$. If 
    $$
    \frac{2c(1-\beta)(1-\alpha\beta)}{\beta\gamma} > 1,
    $$

    then for any $d$, the set $\BBA(c)$ is $(\alpha, \beta)$-losing.
\end{thm}

In the special case of the $\BA$ set, we have the following theorem.

\begin{thm} \label{2bawin}
    Fix $0< c < \frac{1}{6}$. The set $\BA(c)$ is $(\alpha, \beta)$-winning if $\alpha < \frac{1}{3}$ and $\beta > \frac{9c}{1-3\alpha}$.
\end{thm}

Just as with the sets $F_c^-$, we note that some losing parameters for the sets $\BA$ can be found using the Hausdorff dimension of $\BA$. It is not hard to see that for small enough $c$ these losing zones we can obtain this way are contained in the zones given by Theorem \ref{bbalos}.

A special case of winning and losing zones given by Theorem~\ref{bbalos} and Theorem~\ref{2bawin} can be seen at Figure~\ref{Fig:Data7}.

\begin{figure}[!htb]
   \begin{minipage}{0.45\textwidth}
     \centering
     \includegraphics[width=.7\linewidth]{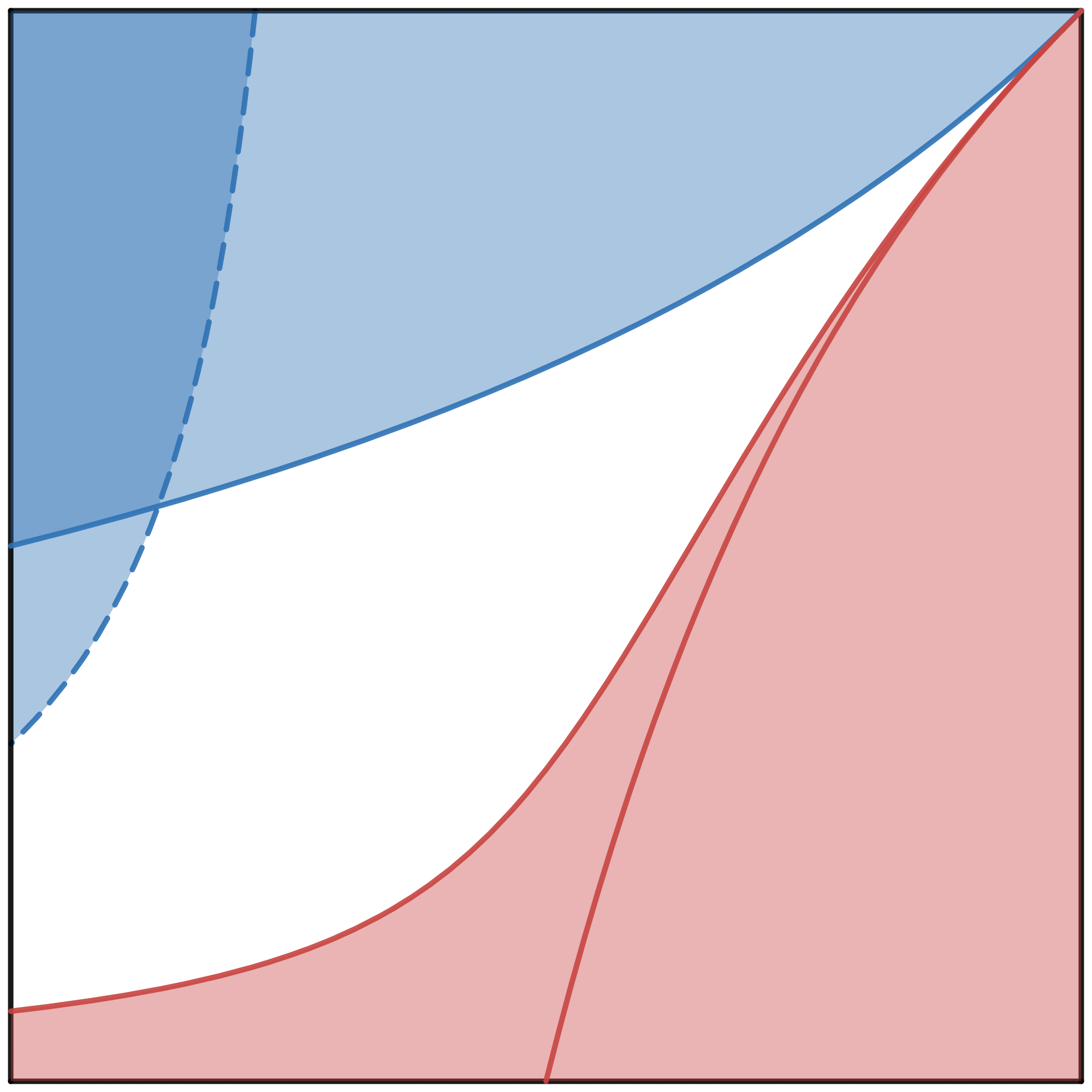}
     \caption{Bounds for $D(\BA(c))$ with $c = 0.035$}\label{Fig:Data7}
   \end{minipage}
\end{figure}

\section{General strategies in the Schmidt game} \label{general}

This section gives some general statements concerning strategies in Schmidt's game. These statements will be used to prove Theorem \ref{thmdf}; however, we believe they may also prove useful for proving results concerning Schmodt games on a larger variety of sets.

Let us start with a classical result due to Schmidt:

\begin{prop}\cite{S}\label{hb}
If a set $S$ is $(\alpha, \beta)$-winning, $\alpha\beta=\alpha'\beta'$, and $\alpha > \alpha'$, then $S$ is also $(\alpha', \beta')$-winning. 

\end{prop}

Sometimes for simplicity we will also refer to existence the of positional strategies, i.e. to the following classical statement:

\begin{prop}\cite{S}
 If a set $S$ is winning, there exists a \textit{positional} winning strategy; that is, a strategy that only depends on the most recent interval played and not the prior history of moves. 
\end{prop}

The next four statements are the technical steps in proving Lemma~\ref{lemmain}; however, they are rather general and can be used separately in the further research.

\begin{lem} \label{lem0}
    Fix $m \in \mathbb{N}$. Let $S \subseteq \mathbb{R}$ be some set with a property that $m^k  = m^k + S = S$ for any integer $k$.
    Suppose Bob has a strategy for $(\alpha, \beta)$-game with the set $S$ that allows him to win, and the first ball in this strategy is $B_0$. Then:

    \begin{enumerate}
        \item For any $k \in \mathbb{Z}$ Bob has a winning strategy for $(\alpha, \beta)$-game with the set $S$ starting with the ball $m^k B_0$;
        \item For any $k \in \mathbb{Z}$ Bob has a winning strategy for $(\alpha, \beta)$-game with the set $S$ starting with the ball $m^k + B_0$.
    \end{enumerate}

    (Here $m^k B_0 = \{m^k x: \,\,\, x \in B_0 \}$ and $m^k + B_0 = \{m^k + x: \,\,\, x \in B_0 \}$).
\end{lem}

\begin{proof}
    \begin{enumerate}
        \item Suppose Bob starts the game with $B_0' = m^k B_0$, and Alice plays some ball $A_0' \subset B_0'$. Let us define $A_0 = m^{-k} A_0' \subseteq B_0$. Now let $B_1 \subset A_0$ be the ball that Bob would play after $A_0$ according to his winning strategy. Let Bob play $B_1' = m^k B_1 \subset A_0'$.

        After $n$ steps, analogously, suppose Alice plays some ball $A_n' \subset B_n'$. Define $A_n = m^{-k} A_n'$, and define $B_{n+1}$ as the ball Bob would play after $A_n$; then let Bob play $B_{n+1}' = m^k B_{n+1} \subset A_n'$.

        It is easy to see that $\bigcap\limits_{i=0}^{\infty} B_i' = m^k \bigcap\limits_{i=0}^{\infty} B_i = m^k x$, where $x$ is the outcome of some game with the starting ball $B_0$, i.e. $x \notin S$, therefore $m^k x \notin S$ and Bob wins.

        \item The strategy and the argument is analogous; we just take $A_i = A_i' - m^k$ and $B_i' = B_i + m^k$ and use the fact that $x \in S$ if and only if $m^k + x \in S$.
    \end{enumerate}
\end{proof}

\begin{lem} \label{lem1}
    Let $S \subseteq \mathbb{R}$ be any set. Suppose Bob has a strategy for $(\alpha, \beta)$-game with the set $S$ that allows him to win, and the first ball in this strategy is $B_0$. Then for any $(\alpha', \beta')$ such that $\beta' < \beta$ and $\alpha' \beta' = \alpha \beta$ Bob has a winning strategy starting from $B_0$ in $(\alpha', \beta')$-game.
\end{lem}

\begin{proof}
    Suppose Bob starts $(\alpha', \beta')$-game with the ball $B_0' = B_0$, and Alice plays some ball $A_0'$. Let $A_0$ be the ball of size $|A_0| = \alpha |B_0| = \frac{\alpha}{\alpha'}|A_0'| < |A_0'|$ centered at the center of $A_0'$; suppose $B_1$ is the ball Bob would play after $A_0$ in $(\alpha, \beta)$-game. Let Bob play the ball $B_1' = B_1$ after $A_0'$ in $(\alpha', \beta')$-game; indeed, $B_1' \subset A_0 \subseteq A_0'$ and $|B_1'| = \beta |A_0| = \alpha \beta |B_0'|$, so Bob can make such a move.

    After $n$ steps, analogously, suppose Alice plays some ball $A_n' \subset B_n'$. Define $A_n$ to be the ball of size $|A_n| = \alpha |B_n| = \frac{\alpha}{\alpha'}|A_n'| < |A_n'|$ centered at the center of $A_n'$, and define $B_{n+1}$ as the ball Bob would play after $A_n$; then let Bob play $B_{n+1}' = B_{n+1} \subset A_n'$. One has $\bigcap\limits_{i=0}^{\infty} B_i' = \bigcap\limits_{i=0}^{\infty} B_i = x$ where $x$ is the outcome of some game with the starting ball $B_0$, i.e. $x \notin S$, therefore Bob wins.

\end{proof}

\begin{lem} \label{lem2}
   Fix $m \in \mathbb{N}$. Let $S \subseteq \mathbb{R}$ be some set with a property that $m^k S = m^k + S = S$ for any integer $k$. Let parameters $(\alpha, \beta)$ be such that $\log_m (\alpha \beta) \notin \mathbb{Q}$, and suppose Bob wins $(\alpha, \beta)$-game with $S$. Then for any $L > 0$ and any $\varepsilon > 0$ there exists a winning strategy in this game for Bob starting with $B_0$ with 

    $$
    L \leq |B_0| < (1 + \varepsilon) L.
    $$
\end{lem}

\begin{proof}
    Suppose the ball Bob starts his winning strategy with has size $l$. We denote $\xi = - \log_{\alpha \beta} m$ and $\chi = \log_{\alpha \beta} (\frac{L}{l})$, and let $\varepsilon' = \log_{\alpha \beta} (1 + \varepsilon)$. 

    As $\xi$ is irrational, the sequence $k \xi + \chi + \varepsilon'$ is uniformly distributed modulo 1, so there exists $k$ such that $0 < \{ k \xi + \chi + \varepsilon' \} \leq \varepsilon'$. Let $n = \lfloor k \xi + \chi + \varepsilon'  \rfloor $; in this case
    $$
    \log_{\alpha \beta} L + k \xi \leq n + \log_{\alpha \beta} l < \log_{\alpha \beta} L + k \xi +  \varepsilon'
    $$
    and
    $$
    \frac{L}{m^k} \leq (\alpha \beta)^n l = |B_n| < (1 + \varepsilon) \frac{L}{m^k},
    $$
    where $B_n$ is any ball Bob can play using his strategy. Fix some possible $B_n$; if it occurs in the game, Bob can start the game with this $B_n$. By Lemma~\ref{lem0} Bob can start also with $B_0 = m^k B_n$ and win. It remains to notice that $|B_0| = m^k |B_n|$ and
    $$
    L \leq |B_0| < (1 + \varepsilon) L.
    $$
\end{proof}

\begin{lem} \label{lem3}
    Let $S \subseteq \mathbb{R}$ be any set. Suppose Bob has a strategy for $(\alpha, \beta)$-game with the set $S$ that allows him to win, and the first ball in this strategy is $B_0$. Then for any $\varepsilon > 0$ and for any $B_0' \subseteq B_0$ with $|B_0'| \geq \frac{|B_0|}{1 + \varepsilon}$, Bob has a winning strategy starting from $B_0'$ in $((1 + \varepsilon) \alpha, \frac{\beta}{1 +\varepsilon})$-game.
\end{lem}

\begin{proof}
    Fix $\varepsilon > 0$. Suppose Bob starts $(\alpha', \beta')$-game (where $\alpha' = (1 + \varepsilon) \alpha$ and $\beta' = \frac{\beta}{1 + \varepsilon}$) with some $B_0'$ such that $|B_0'| = \frac{|B_0|}{1 + \varepsilon}$.

    Suppose Alice chooses some $A_0' \subseteq B_0' \subseteq B_0$; one has $|A_0'| = \alpha' |B_0'| = (1 + \varepsilon) \alpha \frac{|B_0|}{1 + \varepsilon} = \alpha |B_0|$. Let $A_0 = A_0'$, and let $B_1$ be the ball Bob would play after $A_0$ in $(\alpha, \beta)$-game. Now, define $B_1'$ as a ball centered at the center of $B_1$ having $|B_1'| = \frac{|B_1|}{1 + \varepsilon}$, and let Bob play $B_1'$ after $A_0'$ in our $(\alpha', \beta')$-game. 

    Analogously, let $A_n = A_n'$, $B_{n+1}$ the ball Bob would play after $A_n$ in $(\alpha, \beta)$-game and $B_{n+1}'$ a ball centered at the center of $B_{n+1}$ having $|B_{n+1}'| = \frac{|B_{n+1}|}{1 + \varepsilon}$; Bob can play $B_{n+1}'$ after any possible $A_n'$.

    One can see that $\bigcap\limits_{i=0}^{\infty} B_i' = \bigcap\limits_{i=0}^{\infty} B_i = x$ where $x \notin S$, as it is the outcome of an $(\alpha, \beta)$-game where Bob followed his winning strategy. Therefore Bob wins.

    Finally, for an arbitrary $|B_0'| \geq \frac{|B_0|}{1 + \varepsilon}$ let us note that $|B_0'| = \frac{|B_0|}{1 + \varepsilon'}$ for $\varepsilon' \leq \varepsilon$, which means that Bob can win $((1 + \varepsilon')\alpha, \frac{\beta}{1 + \varepsilon'})$-game on $S$ starting from $B_0'$. Lemma~\ref{lem1} completes the proof. 
\end{proof}

\begin{lem} \label{lemmain}
    Fix $m \in \mathbb{N}$. Let $S \subseteq \mathbb{R}$ be some set with a property that $m^k S = m^k + S = S$ for any integer $k$.  Let parameters $(\alpha, \beta)$ be such that $\log_m (\alpha \beta) \notin \mathbb{Q}$, and suppose Bob wins $(\alpha, \beta)$-game with $S$. Then for any $\beta' < \beta, \, \alpha' = \frac{\alpha \beta}{\beta'}$, and any closed ball $T$ Bob can win $(\alpha', \beta')$-game starting with any ball $T$ as his initial ball.
\end{lem}

\begin{proof}
    Fix an arbitrary ball $T$. Also, fix $\varepsilon > 0$, and let $L = |T|$. By Lemma~\ref{lem2} Bob has a winning strategy in $(\alpha, \beta)$-game with a starting ball $B_0$ s.t. $(1 + \frac{\varepsilon}{2}) L \leq |B_0| < (1 + \varepsilon) L$. Take $\varepsilon' < \varepsilon$ such that $|B_0| = (1 + \varepsilon') L$.

    Now, let $k$ be such that $\frac{1}{m^k} < \frac{1}{2} \varepsilon' L = \frac{1}{2}(|B_0| - |T|)$; let $b_0$ be the center of $B_0$ and $t$ the center of $T$. There exists $q \in \mathbb{Z}$ such that $|\frac{q}{m^k} + b_0 - t| < \frac{1}{m^k}$. Consider the ball $B_0' = B_0 + \frac{q}{m^k}$: by Lemma~\ref{lem0}, Bob has a winning strategy in $(\alpha, \beta)$-game with a starting ball $B_0'$. In addition, note that $T \subseteq B_0'$ and $|T| = \frac{|B_0'|}{1 + \varepsilon'}$. By Lemma~\ref{lem3}, Bob has a winning strategy in $((1 + \varepsilon) \alpha, \frac{\beta}{1 + \varepsilon})$-game with a starting ball $T$.

    Now, for $(\alpha', \beta')$ as in Lemma it remains to take $\varepsilon$ under the condition $(1 + \varepsilon) \alpha = \alpha'$.     
\end{proof}

Lemma \ref{lemmain} makes the game "symmetric": now we can assume that both Alice and Bob can not choose their starting balls. This symmetry makes it possible to utilize symmetric strategies:

\begin{lem} \label{symm}
    Let $m \in \mathbb{N}$ and $S, Q \subseteq \mathbb{R}$ be some sets with the following properties:
    \begin{itemize}
        \item $m^k S = m^k + S = S$ for any integer $k$.       
        \item $S \sqcup (- S) \sqcup Q = \mathbb{R}$.
    \end{itemize}

    Then:

    \begin{enumerate}
        \item If $\alpha \geq \beta$, then $(\alpha, \beta)$ is a non-winning point for $S$;
        \item If $\beta > \alpha$ and $\log_m (\alpha \beta) \notin \mathbb{Q}$, then $(\alpha, \beta)$ is a non-losing point for $S \sqcup Q$.
    \end{enumerate}
\end{lem}

\begin{proof}

\begin{enumerate}
    \item Suppose $\alpha \geq \beta$ and $S$ is $(\alpha, \beta)$-winning. In this case Alice has a positional winning strategy described by the function $A(I)$ (where $I \subset \mathbb{R}$ is a closed interval) for this game. Consider $(\beta, \alpha)$-game on the set $S$ and let Bob play the following strategy:

    \begin{itemize}
        \item Bob chooses $B_0 = A([0, 1])$. 
        \item If Alice plays the interval $A_n$ at the $n$-th step, Bob plays $B_{n+1} = - A( - A_n )$.
    \end{itemize}

    Suppose $x = \bigcap\limits_{i=0}^{\infty} B_i$ is the outcome of the game. Then $-x = \bigcap\limits_{i=0}^{\infty} A( - A_n )$ is the outcome of the $(\alpha, \beta)$-game where Bob plays the balls $[-1; 0]; -A_0, -A_1, \ldots $ and Alice plays using her winning strategy $A(I)$. Therefore, $-x \in S$, which means $-x \notin S$ and Bob wins.

    It means that the point $(\alpha, \beta)$ is winning and the point $(\beta, \alpha)$ is losing for the set $S$, which contradicts Proposition~\ref{hb}. Thus, the point $(\alpha, \beta)$ is non-winning.

    \item The idea of the proof is the same as in (a). Suppose $S \sqcup Q$ is $(\alpha, \beta)$-losing; Take $\varepsilon > 0$ under the condition $(1 + \varepsilon)^2 < \frac{\beta}{\alpha}$. Then by Lemma~\ref{lemmain} for any closed ball $T$ Bob has a winning strategy for $\left(\alpha (1 + \varepsilon), \frac{\beta}{(1 + \varepsilon)}\right)$-game starting from $T$. We use this statement to construct a $\left(\frac{\beta}{(1 + \varepsilon)}, \alpha (1 + \varepsilon)\right)$-winning strategy for Alice.

    Suppose Bob starts some $(\frac{\beta}{(1 + \varepsilon)}, \alpha (1 + \varepsilon))$-game with the ball $B_0$. Then Alice picks some ball $A_0 \subseteq B_0$ without any restrictions. If Bob plays the ball $B_n$ on $n$-th step, Alice responds by playing $A_n$ where $-A_n$ is the ball Bob would play in  $(\alpha (1 + \varepsilon), \frac{\beta}{(1 + \varepsilon)})$-game which started from $-A_0$ after the balls $-A_0, -B_1, \ldots, -B_n$ (in this game $-A_i$ are Bob's balls and $-B_i$ are Alice's). Let $x$ be the outcome of this $(\frac{\beta}{(1 + \varepsilon)}, \alpha (1 + \varepsilon))$-game; then $-x$ is the outcome of $(\alpha (1 + \varepsilon), \frac{\beta}{(1 + \varepsilon)})$-game Bob played with his winning strategy, i.e. $-x \in (-S)$ and $x \in S \subseteq S \sqcup Q$. It shows that $(\frac{\beta}{(1 + \varepsilon)}, \alpha (1 + \varepsilon))$ is a winning point for $ S \sqcup Q$.

    But one has $\alpha < \alpha (1 + \varepsilon)  < \frac{\beta}{(1 + \varepsilon)}$, which by Proposition~\ref{hb} contradicts our initial assumption. Therefore $S \sqcup Q$ is $(\alpha, \beta)$-non-losing.    
\end{enumerate}

\end{proof}

\begin{remark}
    In general, it's not guaranteed that the game on $S$ as in Lemma \ref{symm} is determined, so we can only prove the weaker statements about $(\alpha, \beta)$ being non-winning or non-losing. Indeed: by ergodicity of the multiplication by $m$ map, one can see that such a $S$ has to have measure zero or be non-measurable. In the latter case, such a $S$ is not compatible with the Axiom of determinacy. It is not clear if there exists such a non-measurable $S$ on which the game is determined; we only were able to construct the examples of such $S$ using the axiom of choice.
\end{remark}

We finish this section with an almost obvious technical statement that we need to prove Proposition \ref{freq12}.

\begin{lem} \label{shiftgen}
    Suppose $E$ is an $(\alpha, \beta)$-winning set and for some $\xi \in \mathbb{R}$ one has $E + \xi \subseteq S$. Then the set $S$ is $(\alpha, \beta)$-winning.
\end{lem}

\begin{proof}
    If $A(I)$ was Alice's winning positional strategy on the set $E$, then $A(I - \xi) + \xi$ is a winning positional strategy on the set $S$. 
\end{proof}

\section{Proof of Theorem \ref{thmdf}}\label{freqproofs}

 We start with the following classical statement about uniform distribution:
\begin{lem}\cite{B}\label{equidist} If $\alpha$ is irrational, then the sequence $\alpha, 2\alpha, 3\alpha, \ldots$ is uniformly distributed modulo $1$.
\end{lem}

We will prove two separate statements, which together imply the result of Theorem \ref{thmdf}: Proposition \ref{dfgeneral} and Proposition \ref{freq12}.

\begin{prop}\label{dfgeneral}
Let $c \in (0, 1)$, and suppose  $\log_2{\alpha\beta} \notin \mathbb{Q}$. Then, the sets $F^{\pm}_c$ are $(\alpha, \beta)$-winning if $\log(4\alpha)/\log(\alpha\beta) > c$, and $(\alpha, \beta)$-losing if $\log(\alpha/4)/\log(\alpha\beta) \leq c$.
\end{prop}

\begin{proof}
    First, let us notice that our winning conditions imply that $\alpha < \frac{1}{4}$, and the losing conditions imply that $\beta \leq \frac{1}{4}$. We will show the winning strategy under these conditions for Alice in the $(\alpha, \beta)$-game on the set $F^-_c$ (and thus on $F^+_c$, since $F^-_c \subseteq F^+_c$).
    
    Let $r_0$ be the radius of $B_0: \,\, r_0 = \frac{1}{2} |B_0|$, and suppose Bob played some interval $B_t$. Define $k_t = -\floor{\log_2{(\alpha\beta)^t r_0}}$. Consider the smallest $m \in \mathbb{Z}$ such that $m / 2^{k_t} \in B_t$. For such an $m$, one has $\left[\frac{m}{2^{k_t}}, \frac{m+1}{2^{k_t}}\right] \subset B_t$. Then, Alice plays $A_t = \left[\frac{m}{2^{k_t}}, \frac{m}{2^{k_t}} + 2\alpha(\alpha\beta)^tr_0\right]$. Define $l_t$ such that $k_t+l_t = -\floor{\log_2{\alpha(\alpha\beta)^tr_0}}-2$. Then, Alice ensured by her move that for any $x \in A_t$, the $(k_t+1)$-th to the $(k_t+l_t)$-th digits are equal to 0.  In other words, she guarantees $l_t$ zeroes on this turn. For convenience, let $a = \log_2 \alpha$, $b = \log_2 \beta$, and $\rho = \log_2 r_0$; note that $a+b \notin  \mathbb{Q}$, and $a + 2 < 0$. We have
\begin{align*}
l_t &= (k_t+l_t)-k_t \\
&= \floor{(a+b)t+\rho} - \floor{a+(a+b)t+\rho} - 2 \\
& = (a+b)t+\rho - (a+(a+b)t+\rho) - \{(a+b)t+\rho\} + \{a+(a+b)t+\rho\} - 2\\
&= -(a+2) + \{a+(a+b)t+\rho\} - \{(a+b)t+\rho\}. 
\end{align*}

This implies that by the Lemma \ref{equidist},
\begin{align*}
\lim\limits_{n \rightarrow \infty} \frac{\sum_{i=0}^{n-1} l_i}{n}  &= -(a+2).
\end{align*}

Now, let $x = x_0 + \sum\limits_{i=1}^{\infty} x_i 2^{-i} = \bigcap\limits_{t = 0}^{\infty} A_t$ be the outcome of the game. We notice that for $k_n \leq k < k_{n+1}$ one has 

$$
\frac{\# \{ 1\leq i \leq k : x_i = 0 \}}{k} = \frac{\# \{ 1\leq i \leq k_n : x_i = 0 \}}{k_n} \cdot \frac{k_n}{k} + \frac{\# \{ k_n + 1 \leq i \leq k : x_i = 0 \}}{k},
$$
which implies
$$
\frac{\# \{ 1\leq i \leq k_n : x_i = 0 \}}{k_n} \cdot \frac{k_n}{k_{n+1}} \leq \frac{\# \{ 1\leq i \leq k : x_i = 0 \}}{k}.
$$

It is easy to see that $k_{n+1} - k_n$ is bounded, therefore $\lim\limits_{n \rightarrow \infty} \frac{k_n}{k_{n+1}} = 1$, and 

\begin{align*}
d^-(x, 0) &= \liminf\limits_{k \rightarrow \infty} \frac{\# \{ 1\leq i \leq k : x_i = 0 \}}{k} \\
&\geq \liminf\limits_{n \rightarrow \infty} \frac{\# \{ 1\leq i \leq k_n : x_i = 0 \}}{k_n} \cdot \frac{k_n}{k_{n+1}} \\
&\geq \lim\limits_{n \rightarrow \infty} \frac{\sum_{i=0}^{n-1} (l_i)}{1-\log_2((\alpha\beta)^nr_0)} \cdot \frac{k_n}{k_{n+1}} \\
&= \frac{-(a+2)}{-(a+b)}\\
&= \frac{\log(4\alpha)}{\log(\alpha\beta)} > c,
\end{align*}
which means that $x \in F^-_c$. Therefore, Alice wins. The losing condition is analogous, so we omit it.

\end{proof}

For the next statement, we fix $c = \frac{1}{2}$. Together with the proposition above, it clearly implies the statement of Theorem \ref{thmdf} due to trivial inclusions: $F_{c_1}^{\pm} \subseteq F_{c_2}^{\pm}$ if $c_1 \geq c_2$.

\begin{prop} \label{freq12}
    \begin{enumerate}
        \item\label{thm_fr_1} If $\alpha \geq \beta$, then the set $F^-_{\frac{1}{2}} = \{ x: \, d^-(x, 0) > \frac{1}{2} \}$ is $(\alpha, \beta)$-losing;
        \item\label{thm_fr_2} If $\beta > \alpha$ and $\log_2 (\alpha \beta) \notin \mathbb{Q}$, then the set $ \{ x: \, d^+(x, 0) \geq \frac{1}{2} \}$ is $(\alpha, \beta)$-winning.
    \end{enumerate}
\end{prop}

In order to show Proposition \ref{freq12}, we will need the following observation:

\begin{lem} \label{shift}
    Let $\mathfrak{S} \subseteq \mathbb{N}$ be an infinite set of density zero, and define $\xi: = \sum\limits_{j \in \mathfrak{S}} \frac{1}{2^j}$. Suppose $d^{\pm}(x, 0) = c$; then 
    $$
    d^{\pm}\left(x + \xi, 0\right) \geq c.
    $$
\end{lem}

\begin{proof} For $x \in \mathbb{Z}[\frac{1}{2}]$ one has $d^{\pm}t(x + \sum\limits_{j \in \mathfrak{S}} \frac{1}{2^j}, 0) = d^{\pm}(\sum\limits_{j \in \mathfrak{S}} \frac{1}{2^j}, 0) = 1$, as addition of $x$ changes only finitely many digits. Now assume $x \notin \mathbb{Z}[\frac{1}{2}]$.

    We define $Q \subseteq \mathbb{N}$ by the condition $x = \sum\limits_{j \in Q} \frac{1}{2^j}$, and let $\mathfrak{S} = \mathfrak{S}_1 \sqcup \mathfrak{S}_2$ where $\mathfrak{S}_1 = \mathfrak{S} \cap Q$. First, let us show that $d^{\pm} (x + \sum\limits_{j \in \mathfrak{S}_2} \frac{1}{2^j}, 0) = c$. One can see that $x + \sum\limits_{j \in \mathfrak{S}_2} \frac{1}{2^j} = \sum\limits_{j \in Q \sqcup \mathfrak{S}_2} \frac{1}{2^j}$, and (using the notation $\lim^+$ for $\limsup$ and $\lim^-$ for $\liminf$ locally for convenience)

    $$
    d^{\pm}(x + \sum\limits_{j \in \mathfrak{S}_2} \frac{1}{2^j}, 0) = {\lim\limits_{k \rightarrow \infty}}^{\pm}\frac{\# \{ 1 \leq i \leq k: \, i \notin  Q \sqcup \mathfrak{S}_2 \} }{k} = 
    $$
    $$
    {\lim\limits_{k \rightarrow \infty}}^{\pm} \Big( \frac{\# \{ 1 \leq i \leq k: \, i \notin  Q \} }{k} - \frac{\# \{ 1 \leq i \leq k: \, i \in  \mathfrak{S}_2 \} }{k} \Big) = 
    $$
    $$
    {\lim\limits_{k \rightarrow \infty}}^{\pm}\frac{\# \{ 1 \leq i \leq k: \, i \notin  Q \} }{k} - \lim\limits_{k \rightarrow \infty} \frac{\# \{ 1 \leq i \leq k: \, i \in  \mathfrak{S}_2 \} }{k}  = 
    $$
    $$
    = {\lim\limits_{k \rightarrow \infty}}^{\pm}\frac{\# \{ 1 \leq i \leq k: \, i \notin  Q \} }{k} - 0 =  d^{\pm}(x, 0) = c.
    $$

    Now it remains to prove the statement only for $\mathfrak{S}$ such that $\mathfrak{S} \subseteq Q$. Let us define $R \subseteq \mathbb{N}$ by $\sum\limits_{j \in Q} \frac{1}{2^j} + \sum\limits_{j \in \mathfrak{S}} \frac{1}{2^j} = \sum\limits_{j \in R} \frac{1}{2^j}$. It is easy to check that:

    \begin{itemize}
        \item Suppose $i \notin Q$ and $i+1 \notin Q$. Then $i \notin R$;
        \item Suppose $a \notin Q, \, a+b \notin Q$ and $i \in Q$ for all $a < i < a+b$. Then $\#\{ a \leq i < a+b: \, i \in R \} \leq b-1 = \#\{ a \leq i < a+b: \, i \in Q \}$.
    \end{itemize}

    This immediately implies that $d^{\pm}(x + \sum\limits_{j \in \mathfrak{S}} \frac{1}{2^j}, 0) \geq d^{\pm}(x, 0) = c$ and completes the proof.
\end{proof}

Let us also recall simple properties of digit frequencies:

    \begin{enumerate}
        \item \label{prop1} $d^{\pm}(2^k x, i) = d^{\pm}(2^l + x, i)$ for any integers $k, l$;
        \item \label{prop2} $d^{\pm}(x, i) + d^{\mp}(x, 1-i) = 1$;
        \item \label{prop3} If $x \notin \mathbb{Z}[\frac{1}{2}]$, then $d^{\pm}(x, i) = d^{\pm}(-x, 1-i)$.
    \end{enumerate}

\vskip 0.2 cm

We are ready to prove Proposition \ref{freq12}.

\textit{Proof of Proposition \ref{freq12}.} Let $S = F_{\frac{1}{2}^-} \setminus \mathbb{Z}[\frac{1}{2}]$. Then, using the properties \ref{prop2} and \ref{prop3} listed above, it is easy to verify that $-S = \{ x: \,\, d^{+}(-x, 0) < \frac{1}{2} \}\setminus \mathbb{Z}[\frac{1}{2}]$ is a set disjoint with $S$. Let us define $Q$ under the condition  $S \sqcup (- S) \sqcup Q = \mathbb{R}$; in this case, $S \sqcup Q =  \{ x: \, d^+(x, 0) \geq \frac{1}{2} \} \cup \mathbb{Z}[\frac{1}{2}]$. 

Property \ref{prop1} shows that the sets $S$ and $Q$ satisfy the conditions of the Lemma \ref{symm} with $m = 2$. As in addition our sets are Borel, the $(\alpha, \beta)$-game on them is determined, and we can claim that 
\begin{itemize}
        \item If $\alpha \geq \beta$, then $(\alpha, \beta)$ is a losing point for $S$;
        \item If $\beta > \alpha$ and $\log_2 (\alpha \beta) \notin \mathbb{Q}$, then $(\alpha, \beta)$ is a winning point for $S \sqcup Q$.
    \end{itemize}
    
    Now let us fix $\xi$ as in Lemma \ref{shift}: $\xi$ is irrational, so by the aforementioned lemma $F_{\frac{1}{2}}^- +\xi \subseteq S$ 
 and $S \sqcup Q + \xi \subseteq \{ x: \, d^+(x, 0) \geq \frac{1}{2} \}$.  The following two observations complete the proof of Proposition \ref{freq12}:

 \begin{enumerate}
     \item Suppose $F_{\frac{1}{2}}^-$ is $(\alpha, \beta)$-winning for $\alpha \geq \beta$; then, by Lemma \ref{shiftgen}, $S$ should also be $(\alpha, \beta)$-winning, which is not true. Provided that the game is determined on $F_{\frac{1}{2}}^-$, we conclude that the set $F_{\frac{1}{2}}^-$ is $(\alpha, \beta)$-losing.

     \item Follows directly from Lemma \ref{shiftgen} and observations above.
 \end{enumerate} 
 \hfill \qedsymbol{}






\section{Proofs of the statements regarding d-Badly Approximable Numbers} \label{2ba}

We begin results in this section with a useful general lemma:

\begin{lem}\label{escape}
    Suppose Bob plays any interval $B_t$ and $\gamma>0$. Denote the center of $B_t$ by $c_t$. Then, Bob can always choose an interval $B_{t+1}$ such that $|c_{t+1} - c_{t}| \leq \frac{|B_t|}{2} \gamma$. As a corollary, Bob can guarantee that for any $m \in \mathbb{N}$, $|c_{t+m}-c_t| \leq \frac{|B_t|}{2(1-\alpha\beta)}\gamma$. Analogous results hold for Alice, with $\gamma$ replaced by $1 - 2 \beta + \alpha \beta$. 

\end{lem}

\begin{proof} We prove Bob's case; Alice's case is analogous.

Let's denote center of $B_{t}$ by $c_t$. It is clear that Bob can always choose such a $B_{t+1}$ that $|c_{r+1} - c_{r}| \leq \frac{|B_t|}{2} \gamma$, since no matter which side Alice picks her interval on, Bob will go the opposite way. Extending this logic, it follows that $|c_{m+r} - (a_0 + k \delta)| \leq \frac{\rho (\alpha \beta)^m}{2} \gamma (1 + (\alpha \beta) + (\alpha \beta)^2 + \ldots) = \frac{\rho (\alpha \beta)^m}{2 (1 - \alpha \beta)} \gamma$, which completes the proof.

\end{proof}

Now we are ready to give the proofs of our main theorems.

\textit{Proof of Theorem \ref{bbalos}}
    
One has 

$$
\BBA(c)^c = \left\{ x: \,\, \forall N \,\, \exists \, n > N \,\, \text{and} \,\, k \, \in \, \mathbb{Z}: \,\, \left|x - \frac{k}{d^n}\right| \leq \frac{c}{d^n} \right\} = 
$$
$$
= \bigcap\limits_{N = 1}^{\infty} \bigcup\limits_{n=N}^{\infty} \left( \bigcup\limits_{k \in \mathbb{Z}} \left[\frac{k}{d^n} - \frac{c}{d^n}, \frac{k}{d^n} + \frac{c}{d^n}\right] \right) \subseteq \bigcap\limits_{N = 1}^{\infty} \bigcup\limits_{n=N}^{\infty} \left( \bigcup\limits_{k \in \mathbb{Z}} \left(\frac{k}{d^n} - \frac{c}{d^n}, \frac{k}{d^n} + \frac{c}{d^n}\right) \right), 
$$

and let $J_n = \bigcup\limits_{k \in \mathbb{Z}} \left(\frac{k}{d^n} - \frac{c}{d^n}, \frac{k}{d^n} + \frac{c}{d^n}\right)$. Fix $\e>0$ such that
$$
\e = \frac{2c(1-\beta)(1-\alpha\beta)}{\beta\gamma} - 1.
$$
Bob will play in the following way:

\begin{itemize}
    \item On his first turn, Bob will play an interval of length 
    $$
    \frac{1}{d} \left(\frac{2c(1-\alpha\beta)}{\gamma}\right)
    $$
    centered at $1/d$. By Lemma 1.2, Bob can play in such a way that in finitely many moves, $B_{s} \subset J_1$.
    \item On all subsequent turns, Bob plays arbitrarily until some turn $t$ such that
    \begin{align}
    \log_{d}\left(\frac{1}{1-\beta}\right) &\leq \left\{\left(t \cdot \log_d \alpha\beta\right) + \log_d |A_0| \right\} \label{a0}\\
    &\leq \log_{d}\left(\frac{1}{1-\beta}\right) + \log_d(1+\e), \nonumber
    \end{align}
    there exists $N \in \mathbb{N}$ such that
    \begin{equation}\label{a}
    \left(\frac{1}{1-\beta}\right) \leq d^N|A_t| \leq \left(\frac{1}{1-\beta}\right)(1+\e) = \left(\frac{2c(1-\alpha\beta)}{\beta\gamma}\right), 
    \end{equation}
    and since $|A_t|(1-\beta) > 1/d^N$, Bob can play his next interval centered at $m/d^N$ for some $m \in \mathbb{Z}$. Furthermore, by Lemma 1.2, Bob can play in such a way that in finitely many moves, $B_{t+s} \subset J_n$. By Lemma \ref{equidist}, there exists an infinite number of possible $t$ that satisfy inequality \ref{a}. Therefore, Bob can do this an infinite number of times, and Alice loses.

\end{itemize}
 \hfill \qedsymbol{}

\textit{Proof of Theorem \ref{2bawin}}
We will describe the winning strategy for Alice. Suppose on the step $t$ of the game Bob played an interval $B_t$ of length $l_t$. Pick $k_t$ such that 
\begin{equation}\label{ineqlt}
\frac{l_t}{2} \leq \frac{1}{2^{k_t}} < l_t \leq \frac{1}{2^{k_t-1}}.
\end{equation}

Alice's strategy is to choose a ball of length $\alpha l_t$ and center $c_t$ of the form $(r \pm \frac{1}{3})2^{-k_t}$ for some $r \in \Z$ as $A_t$. She can always do this, as for any two subsequent points $c_t^1, c_t^2$ of the chosen form one has
$$
\left|c^1_t - c^2_t\right| + \alpha l_t = \frac{1}{3 \cdot 2^{k_t - 1}} + \alpha l_t < \frac{2l_t}{3} + \frac{l_t}{3} = l_t.
$$

In some special cases, when needed, we will specify which of such balls Alice will need to pick.

Define $q_t$ to be the smallest integer such that $m/2^{q_t} \in A_t$ for some $m \in \Z$. Note that $q_t > k_t$, as no points of the form $\frac{s}{2^{k_t}}$ with $s \in \mathbb{Z}$ belong to $A_t$ by construction. One has
\begin{equation} \label{ineqqt}
    \left|\frac{m}{2^{q_t}} - c_t\right| =\frac{1}{3\cdot2^{q_t}} < \frac{\alpha l_t}{2} < \frac{1}{3\cdot2^{q_t-1}}.
\end{equation}

We wish to show that for all $i \geq 1$ and $x \in A_i$,
$$
\left|x-\frac{m}{2^n}\right| > \frac{c}{2^n} \ \text{for all} \ n \leq q_i-2
$$
by induction on the number of step $i$. Suppose that the statement holds for $i$, and Bob picks an interval $B_{i+1}$ with length $l_{i+1}=\alpha\beta l_i$. 

\begin{itemize}
    \item First, we consider $q_i \leq n \leq k_{i+1}$. Note that for all $x \in A_{i+1}$ and any $m \in \mathbb{Z}$
\begin{align*}
\left|x-\frac{m}{2^n}\right| \geq \left|x-\frac{m'}{2^{k_{i+1}}}\right| &> \frac{1}{3\cdot2^{k_{i+1}}} - \frac{\alpha l_{i+1}}{2} \\
&\geq \frac{l_{i+1}}{6} - \frac{\alpha l_{i+1}}{2} \\
&= \alpha \beta l_i \left(\frac{1-3\alpha}{6}\right) \\
&> \alpha l_i \left(\frac{9c}{1-3\alpha}\right)\left(\frac{1-3\alpha}{6}\right) \\
&= \frac{3}{2} \alpha l_ic \\
&> \frac{c}{2^{q_i}} \geq \frac{c}{2^n},
\end{align*}
where we used \ref{ineqqt} in the second to last inequality. 

\item To take care of $k_{i+1}+1 \leq n \leq q_{i+1}-2$, simply note that for any $x \in A_{i+1}$ and any $c<1/6$,
\begin{align*}
\left|x-\frac{m}{2^n}\right| &> \frac{1}{3\cdot2^n} - \frac{\alpha l_{i+1}}{2} > \frac{1}{3\cdot2^n} - \frac{1}{3 \cdot 2^{q_{i+1}-1}} > \frac{1}{6 \cdot 2^n} \geq \frac{c}{2^n}.
\end{align*}

Here, we again used \ref{ineqqt} in the second inequality.

\item Finally, consider $n=q_i-1$. There are two possible cases based on $A_i$, the prior move.
\begin{itemize}\item[\rm \textbf{Case I.}] If
$$
\frac{\alpha l_i}{2} < \frac{1}{3 \cdot 2^{q_i-1}} - \frac{c}{2^{q_i-1}},
$$
then for all $x \in A_i$ and $m \in \mathbb{Z}$
\begin{align*}
\left|x-\frac{m}{2^{q_i-1}}\right| &\geq \frac{1}{3\cdot2^{q_i-1}} - \frac{\alpha l_i}{2} > \frac{c}{2^{q_i-1}},
\end{align*}

which means that the desired inequality is already satisfied for any $x \in A_i$ and thus for any $x \in A_{i+1}$.

\item[\rm \textbf{Case II.}] Assume
$$
\frac{\alpha l_i}{2} \geq \frac{1}{3 \cdot 2^{q_i-1}} - \frac{c}{2^{q_i-1}}.
$$
If 
$$
B_{i+1} \cap \left[\frac{m-c}{2^{q_i-1}}, \frac{m+c}{2^{q_i-1}}\right] = \emptyset
$$
for any $m \in \mathbb{Z}$, it is clear that Alice does not have to worry about $n=q_i-1$, since for all $x \in A_{i+1} \subset B_{i+1}$, $|x-m/2^{q_i-1}|>c/2^{q_i+1}$. Thus, for the remainder of this case, assume  $B_{i+1}$ intersects such an interval.
Without loss of generality (the other case is symmetric), assume also that $m/2^{q_i-1} > x$ for all $x \in B_{i+1}$ (that is, the fraction lies to the right of the interval $B_{i+1}$). Alice will then choose an interval $A_{i+1}$ according to her strategy which is centered at $c_{i+1}$ such that
$$
c_{i+1} \leq \frac{m}{2^{q_i-1}}-\frac{2}{3}\left(\frac{1}{2^{k_{i+1}}}\right),
$$
that is, Alice can not choose the closest interval to $\frac{m}{2^{q_i-1}}$. This is always possible because the second closest interval of the desired form is within $l_{i+1}$ of $\frac{m}{2^{q_i-1}}$. Indeed: the left endpoint of the second closest interval of this form is 
$$e_{i+1} = \frac{m}{2^{q_i-1}} - \frac{2}{3}\left(\frac{1}{2^{k_{i+1}}}\right) - \frac{\alpha l_{i+1}}{2}, $$

while the left endpoint of $B_{i+1}$ is smaller than $\frac{m}{2^{q_i-1}} - l_{i+1} < e_i$, where the latter inequality holds as 
$$
\frac{2}{3}\left(\frac{1}{2^{k_{i+1}}}\right) + \frac{\alpha l_{i+1}}{2} < \frac{2}{3}l_{i+1} + \frac{1}{6}l_{i+1} < l_{i+1}
$$

by \ref{ineqlt}.

In this case, for $x \in A_{i+1}$ one has
\begin{align*}
\left|x-\frac{m}{2^{q_i-1}}\right| &> \frac{2}{3}\left(\frac{1}{2^{k_{i+1}}}\right)  - \frac{\alpha l_{i+1}}{2}\\
&> \frac{l_{i+1}}{3} - \alpha l_{i+1} \\
&= \alpha\beta l_i \left(\frac{1-3\alpha}{3}\right) \\
&> 3\alpha l_i c \\
&> \frac{c}{2^{q_i-1}},
\end{align*}
proving the desired inequality in this last case. This completes the induction, and thus the whole proof. 
\end{itemize} 

\end{itemize}

\hfill \qedsymbol{}

\section*{Acknowledgements}
We are very grateful to Prof.~Dmitry Kleinbock for the project proposal and his valuable suggestions in research paths. We would like to thank Prof.~Lior Fishman and Prof.~Stephen Jackson for fruitful discussions of the Schmidt game determinacy problems. We would also like to thank Prof.~Pavel Etingof, Dr.~Slava Gerovitch, Dr.~Felix Gotti, Dr.~Tanya Khovanova, and all the PRIMES organizers for providing the opportunity to conduct math research through this program.

\end{document}